\documentclass[12pt]{article}
\usepackage{epsfig,amsmath,amssymb,latexsym,url}

\newcommand{\RR}{\mathbb{R}}

\def\bft{\mathbf{t}}

\def\bfx{\mathbf{x}}

\def\bfv{\mathbf{v}}
\def\bfy{\mathbf{y}}

\def\bfeps{{\boldsymbol \epsilon}}

\newtheorem{theorem}{Theorem}
\newtheorem{lemma}{Lemma}

\newtheorem{corollary}{Corollary}

\newenvironment{proof}{\begin{trivlist}\item[]{\emph{Proof.}}}
               {\hfill$\Box$\end{trivlist}}

\def\Acknowledgement{\goodbreak\bigskip\noindent{\bf Acknowledgement.\ }}

\begin{document}

\title{A divided difference identity \\
      for a class of multiple integrals}
\author{
Michael S. Floater\footnote{
Department of Mathematics,
University of Oslo,
PO Box 1053, Blindern,
0316 Oslo,
Norway,
{\it email: michaelf@math.uio.no}}
}                                                                               
                                                                                
\maketitle

\begin{abstract}
We derive an identity that relates a class of
multiple integrals involving Vandermonde polynomials to divided differences.
Alternatively the identity can be viewed as an integral
formula for divided differences.
As part of the derivation we show that
both sums of pure partial derivatives and mixed partial derivatives of
Vandermonde polynomials are zero,
which might be of independent interest.
\end{abstract}


\smallskip

\noindent {\em Math Subject Classification: }
15A15, 
65D05, 
65F40. 

\section{Introduction}
In this paper we show that there is a surprisingly simple formula
for a certain type of multiple integral involving
Vandermonde polynomials.
This formula arose while working on bounds for the determinant
of an exponential matrix, and will be used for that
purpose in \cite{Floater:2026b}.
Here, we will focus purely on a proof of the formula
(Theorem~\ref{thm:ddid}).
The formula can alternatively be viewed as an integral
formula for a divided difference (at distinct points)
which is different to the
Genocchi-Hermite formula and to the B-spline integral formula of
Curry and Schoenberg~\cite{CurrySchoenberg:1966}.
For a survey of divided differences and their role
in interpolation theory, see~de Boor~\cite{deBoor:2005},
and for some basic theory, see
Isaacson and Keller~\cite[Section 6.1]{Isaacson:1994}
and Steffensen~\cite{Steffensen:1927}.

For some $n \ge 1$, let $\bfx = (x_1,x_2,\ldots,x_{n+1})$ be a real increasing
sequence and consider the rectangle in $\RR^n$,
\begin{equation}\label{eq:seq_rect}
 R(\bfx) := [x_1,x_2] \times [x_2,x_3] \times \cdots 
	     \times [x_n,x_{n+1}].
\end{equation}
Let $\phi(\bft) = \phi(t_1,t_2,\ldots,t_n)$ be a 
continuous function $\phi : R(\bfx) \to \RR$. Then its integral
over $R(\bfx)$
is
$$ \int_{R(\bfx)} \phi(\bft) \, d\bft
 := \int_{x_n}^{x_{n+1}} \cdots \int_{x_2}^{x_3} \int_{x_1}^{x_2}
 \phi(t_1,t_2,\ldots,t_n) \, dt_1 \, dt_2 \cdots dt_n. $$
Let us denote by $V(\bft)$ the Vandermonde polynomial
\begin{equation}\label{eq:vandpoly}
V(\bft) = V(t_1,t_2,\ldots,t_n) := \prod_{1 \le i < j \le n} (t_j-t_i).
\end{equation}
We note that Vandermonde polynomials
appear in the study of generalized Vandermonde
determinants and Schur
polynomials~\cite{Macdonald:1995,deMarchi:2001,Ait-Haddou:2019}.
Let
\begin{equation}\label{eq:yidef}
y_i := x_1 + \cdots + x_{n+1-i} + x_{n+3-i} + \cdots + x_{n+1},
\quad i=1,2,\ldots,n+1.
\end{equation}
Thus, $y_i$ is the sum of the $x_j$ except the `$i$-th to last',
$x_{n+2-i}$,
and we can alternatively write
\begin{equation}\label{eq:yialt}
y_i = \Big(\sum_{j=1}^{n+1} x_j \Big) - x_{n+2-i}.
\end{equation}
Thus,
\begin{equation}\label{eq:yminmax}
\begin{aligned}
y_1 &= x_1 + \cdots + x_n, \\
y_{n+1} &= x_2 + \cdots + x_{n+1},
\end{aligned}
\end{equation}
and by (\ref{eq:yialt}),
$$ y_{i+1} - y_i = x_{n+2-i} - x_{n+1-i} > 0 $$
and so $y_1,\ldots,y_{n+1}$ are increasing.
It also follows from (\ref{eq:yminmax}) that for $\bft \in R(\bfx)$,
\begin{equation}\label{eq:tsumbound}
 y_1 \le t_1 + t_2 + \cdots + t_n \le y_{n+1}.
\end{equation}
We will show
\begin{theorem}\label{thm:ddid}
	For $f \in C^n[y_1,y_{n+1}]$,
\begin{equation}\label{eq:ddid}
\int_{R(\bfx)} V(\bft) f^{(n)}(t_1 + t_2 + \cdots + t_n)\, d\bft
	= V(\bfx) [y_1,y_2,\ldots,y_{n+1}]f,
\end{equation}
where
$[y_1,y_2,\ldots,y_{n+1}]f$ denotes the divided difference
of $f$ at the points $y_1,y_2,\ldots,y_{n+1}$.
\end{theorem}

For example, for $n=1$,
$$ \int_{x_1}^{x_2} f'(t_1) \, dt_1
= (x_2-x_1) [y_1,y_2]f
= (x_2-x_1) [x_1,x_2]f, $$
and for $n=2$,
$$ \int_{x_2}^{x_3} \int_{x_1}^{x_2}
(t_2-t_1)
f''(t_1+t_2) \, dt_1 \, dt_2
= (x_2-x_1) (x_3-x_1)(x_3-x_2) [y_1,y_2,y_3]f, $$
where
$$ y_1 := x_1 + x_2, \quad y_2 := x_1 + x_3,  \quad y_3 := x_2 + x_3. $$

In \cite{Floater:2026b}, we will use both Theorem~\ref{thm:ddid} and the
following corollary.
Letting $f(\alpha) := \alpha^n/n!$ shows:
\begin{corollary}\label{cor:vandint}
$$ \int_{R(\bfx)} V(\bft) \, d\bft = \frac{1}{n!} V(\bfx).  $$
\end{corollary}
For example, in the case $n=2$,
$$ \int_{x_2}^{x_3} \int_{x_1}^{x_2} (t_2-t_1) \, dt_1 \, dt_2
= \frac{1}{2} (x_2-x_1) (x_3-x_1)(x_3-x_2). $$

\section{A divided difference viewpoint}
The identity of Theorem~\ref{thm:ddid} can be
expressed purely as a formula for the divided difference
$[y_1,y_2,\ldots,y_{n+1}]f$ for given
$y_1 < y_2 < \cdots < y_{n+1}$.
To see this observe first that from (\ref{eq:yialt}), 
$$ y_j - y_i = x_{n+2-i} - x_{n+2-j}, $$
and so $V(\bfx) = V(\bfy)$, where $\bfy := (y_1,y_2,\ldots,y_{n+1})$.
Next observe that treating $\bfx$ and $\bfy$ as column vectors,
they are related by the matrix equation $\bfy = M \bfx$, where
$M$ is the $(n+1) \times (n+1)$ matrix
$$
M :=
\left[
\begin{matrix}
1 & 1 & \cdots & 1 & 0 \\
1 & 1 & \cdots & 0 & 1 \\
\vdots & \vdots & & \vdots & \vdots \\
1 & 0 & \cdots & 1 & 1 \\
0 & 1 & \cdots & 1 & 1
  \end{matrix}
  \right].
  $$
It is easy to check that $M$ has the inverse
$$
M^{-1} = \frac{1}{n}
\left[
\begin{matrix}
1 & 1 & \cdots & 1 & 1-n \\
1 & 1 & \cdots & 1-n & 1 \\
\vdots & \vdots & & \vdots & \vdots \\
1 & 1-n & \cdots & 1 & 1 \\
1-n & 1 & \cdots & 1 & 1
\end{matrix}
\right],
$$
and therefore,
\begin{equation}\label{eq:xidef}
x_i = \frac{1}{n} 
\Big(\sum_{j=1}^{n+1} y_j - n y_{n+2-i} \Big),
\quad i=1,2,\ldots,n+1.
\end{equation}
So, by Theorem~\ref{thm:ddid},
$$ [y_1,y_2,\ldots,y_{n+1}]f = \frac{1}{V(\bfy)}
\int_{R(\bfx)} V(\bft) f^{(n)}(t_1 + t_2 + \cdots + t_n) \, d\bft, $$
with $\bfx$ given by (\ref{eq:xidef}).
For example, in the case $n=2$,
$$ [y_1,y_2,y_3]f = \frac{1}{V(\bfy)}
\int_{x_2}^{x_3} \int_{x_1}^{x_2}
(t_2-t_1) f''(t_1+t_2) \, dt_1 \, dt_2, $$
where
$$ V(\bfy) = (y_2-y_1)(y_3-y_1)(y_3-y_2), $$
and
$$ x_1 := \frac{y_1 + y_2 - y_3}{2}, \quad
   x_2 := \frac{y_1 - y_2 + y_3}{2}, \quad
   x_3 := \frac{-y_1 + y_2 + y_3}{2}.
$$

\section{Partial derivatives of $V(\bft)$}
The rest of this paper is devoted to proving Theorem~\ref{thm:ddid}.
A key element is the fact that the sums of mixed partial derivatives
of any fixed order of $V(\bft)$ are zero:
Lemma~\ref{lem:Vmixedderivsums}.
As the first step towards that,
we consider a polynomial of the form
\begin{equation}\label{eq:omega}
 \omega(t) := (t-t_1) (t-t_2) \cdots (t-t_m),
\end{equation}
where $t,t_1,\ldots,t_m \in \RR$.
We want to express its derivatives in terms of products of
the differences $t-t_i$.

\subsection{Derivatives of $\omega$}
To compute derivatives of $\omega$ in (\ref{eq:omega}),
consider the elementary symmetric polynomials.
These are defined, for $t_1,t_2,\ldots,t_m \in \RR$, as
$$ e_0(t_1,\ldots,t_m) := 1, $$
and for $k=1,\ldots,m$,
\begin{equation}\label{eq:elem_seq}
 e_k(t_1,t_2,\ldots,t_m) :=
  \sum_{1 \le i_1 < i_2 < \cdots < i_k \le m} t_{i_1} t_{i_2} \cdots t_{i_k}.
\end{equation}

\begin{lemma}\label{lem:ederivs}
For $k=1,\ldots,m$,
$$
 \frac{d}{dt} e_k(t-t_1,\ldots,t-t_m) =
   (m-k+1) e_{k-1}(t-t_1,\ldots,t-t_m).
$$
\end{lemma}

\begin{proof}
By (\ref{eq:elem_seq}),
$$
e_k(t-t_1,\ldots,t-t_m) =
  \sum_{1 \le i_1 < \cdots < i_k \le m}
	(t-t_{i_1}) \cdots (t-t_{i_k}),
$$
and its derivative with respect to $t$ is
\begin{equation}\label{eq:ekderiv}
  \sum_{1 \le i_1 < \cdots < i_k \le m}
\sum_{r=1}^k
(t-t_{i_1})\cdots(t-t_{i_{r-1}})
(t-t_{i_{r+1}})\cdots(t-t_{i_k}).
\end{equation}
We can rewrite this as a sum over
products of the form
\begin{equation}\label{eq:prod}
(t-t_{j_1})\cdots(t-t_{j_{k-1}})
\end{equation}
where
$$ 1 \le j_1 < \cdots < j_{k-1} \le m. $$
Since there are $m-k+1$ elements in the complement
of $\{j_1,\ldots,j_{k-1}\}$ in $\{1,\ldots,m\}$,
the product (\ref{eq:prod}) appears $m-k+1$
times in the sum in (\ref{eq:ekderiv}) and so
$$
 \frac{d}{dt} e_k(t-t_1,\ldots,t-t_m) =
(m-k+1) \sum_{1 \le j_1 < \cdots < j_{k-1} \le m}
	(t-t_{j_1})\cdots(t-t_{j_{k-1}}).
$$
\end{proof}

This lemma enables us to find the derivatives of $\omega$ in (\ref{eq:omega}).
\begin{lemma}\label{lem:pderivs}
Let $\omega$ be the polynomial in (\ref{eq:omega}). Then
for $k=0,1,\ldots,m$,
$$ \omega^{(k)}(t) = k! \, e_{m-k}(t-t_1,t-t_2,\ldots,t-t_m).  $$
\end{lemma}

\begin{proof}
Since $\omega(t) = e_m(t-t_1,t-t_2,\ldots,t-t_m)$,
applying Lemma~\ref{lem:ederivs} iteratively gives
\begin{align*}
\omega'(t) &= e_{m-1}(t-t_1,t-t_2,\ldots,t-t_m), \cr
\omega''(t) &= 2! e_{m-2}(t-t_1,t-t_2,\ldots,t-t_m), \cr
\omega'''(t) &= 3! e_{m-3}(t-t_1,t-t_2,\ldots,t-t_m),
\end{align*}
and so on.
\end{proof}

\subsection{Pure partial derivatives of $V$}
Using Lemma~\ref{lem:pderivs},
we next derive a formula for the `pure' (not mixed) derivatives
of the Vandermonde polynomial $V(\bft)$ in (\ref{eq:vandpoly})
when $t_1,\ldots,t_n$ are distinct.
\begin{lemma}\label{lem:Vderivs}
Let $V(\bft)$ be the Vandermonde polynomial
of (\ref{eq:vandpoly}) with $t_1,\ldots,t_n$ distinct.
For each $i=1,\ldots,n$ and $k=1,\ldots,n-1$,
$$
\frac{\partial^k}{\partial t_i^k} V(\bft)
= k! \, V(\bft) \,
e_k\Big(\frac{1}{t_i-t_1},\ldots,\frac{1}{t_i-t_{i-1}},
	\frac{1}{t_i-t_{i+1}},\ldots,\frac{1}{t_i-t_n}\Big).
$$
\end{lemma}

\begin{proof}
We can write $V(\bft) = c \omega(t_i)$, where
$$
c := (-1)^{n-i} V(t_1,\ldots,t_{i-1},t_{i+1},\ldots,t_n),
$$
which is independent of $t_i$, and
$$ \omega(t) := (t-t_1) \cdots (t-t_{i-1})
(t-t_{i+1}) \cdots (t-t_n). $$
Then
$$
\frac{\partial^k}{\partial t_i^k} V(\bft)
	= c \omega^{(k)}(t_i)
	= V(\bft) \frac{\omega^{(k)}(t_i)}{\omega(t_i)}.
$$
By Lemma~\ref{lem:pderivs},
$$
\omega^{(k)}(t) = k! \, 
e_{n-1-k}(t-t_1,\ldots,t-t_{i-1},t-t_{i+1},\ldots,t-t_n),
$$
and therefore,
$$
\frac{\omega^{(k)}(t_i)}{\omega(t_i)}
= k!
e_k\Big(\frac{1}{t_i-t_1},\ldots,\frac{1}{t_i-t_{i-1}},
	\frac{1}{t_i-t_{i+1}},\ldots,\frac{1}{t_i-t_n}\Big).
$$
\end{proof}

The lemma covers the cases $k=1,\ldots,n-1$.
Since $V(\bft)$ is a polynomial of degree $n-1$
with respect to $t_i$, it follows that
\begin{equation}\label{eq:Vderivn}
\frac{\partial^n}{\partial t_i^n} V(\bft) = 0.
\end{equation}

\subsection{Sums of pure derivatives}
Next we show that the sums of pure derivatives of $V(\bft)$ are zero.
\begin{lemma}\label{lem:Vderivsums}
Let $V(\bft) := V(t_1,\ldots,t_n)$ for any $t_1,\ldots,t_n \in \RR$.
For each $k=1,\ldots,n$,
$$
\sum_{i=1}^n \frac{\partial^k}{\partial t_i^k} V(\bft) = 0.
$$
\end{lemma}

\begin{proof}
For $k=n$ this is implied by (\ref{eq:Vderivn}).
For $1 \le k \le n-1$
it is sufficient to prove it for distinct $t_1,\ldots,t_n$.
Then, by Lemma~\ref{lem:Vderivs},
$$
\sum_{i=1}^n \frac{\partial^k}{\partial t_i^k} V(\bft)
= k! \, V(\bft) S, $$
where
$$ S :=
\sum_{i=1}^n
\sum_{\substack{1 \le i_1 < \cdots < i_k \le n \\
	i \not\in \{i_1,\ldots,i_k\}}}
	\frac{1}{\prod_{j=1}^k(t_i-t_{i_j})}.
$$
By considering the sequence of length $k+1$
formed by $i$ and $i_1,\ldots,i_k$,
we can rewrite $S$ as
$$ S = \sum_{1 \le j_1 < \cdots < j_{k+1} \le n}
S(j_1,j_2,\ldots,j_{k+1}), $$
where
$$
S(j_1,j_2,\ldots,j_{k+1}) :=
\sum_{r=1}^{k+1}
\frac{1}{\prod_{s=1, s \ne r}^{k+1}(t_{j_r}-t_{j_s})}.
$$
We recognize $S(j_1,j_2,\ldots,j_{k+1})$
as the divided difference of order $k$ of the constant function~$1$
over the $k+1$ points $t_{j_1},\ldots,t_{j_{k+1}}$,
which is zero.
\end{proof}

\subsection{Sums of mixed derivatives}
We now show that the sums of mixed derivatives of $V(\bft)$ are also zero.
For $k=1,\ldots,n$, let $E_k$ be the differential operator
\begin{equation}\label{eq:Ek}
E_k := \sum_{1 \le i_1 < \cdots < i_k \le n}
\frac{\partial^k}{\partial t_{i_1} \cdots \partial t_{i_k}}.
\end{equation}
\begin{lemma}\label{lem:Vmixedderivsums}
Let $V(\bft) = V(t_1,\ldots,t_n)$ for $t_1,\ldots,t_n \in \RR$.
Then $E_k V(\bft) = 0$ for all $k=1,\ldots,n$.
\end{lemma}

\begin{proof}
Let $P_k$ denote the differential operator
$$ P_k := \sum_{i=1}^n \frac{\partial^k}{\partial t_i^k}. $$
We have already shown that $P_k V(\bft) = 0$ for all $k=1,\ldots,n$.
The operators $P_k$ and $E_k$ are analogous to the
power sum and elementary symmetric polynomials and thus
obey Newton's identities:
\begin{equation*}
\begin{aligned}
E_1 &= P_1, \cr
2E_2 &= E_1 P_1 - P_2, \cr
3E_3 &= E_2 P_1 - E_1 P_2 + P_3, \cr
\cdots \cr
k E_k &= \sum_{i=1}^{k-1} (-1)^{i-1} E_{k-i} P_i + (-1)^{k-1} P_k.
\end{aligned}
\end{equation*}
Thus, since $P_1V(\bft) = \cdots = P_kV(\bft) = 0$,
it follows that $E_kV(\bft) = 0$ as well.
\end{proof}

\section{The sum function}

For $\bft =(t_1,\ldots,t_n) \in \RR^n$, let $s: \RR^n \to \RR$ denote the sum
\begin{equation}\label{eq:s}
 s(\bft) := t_1 + \cdots + t_n.
\end{equation}
We note that $s$ has the property that
$$ \frac{\partial s}{\partial t_i} = 1, \qquad i=1,\ldots,n. $$
It then follows from the chain rule that if $f \in C^1(s(\Omega))$ for some
domain $\Omega \subseteq \RR^n$, then
$$ \frac{\partial}{\partial t_i} \big(f(s(\bft))\big)
  = f'(s(\bft)), \qquad i=1,\ldots,n. $$
Building on this fact, we next compute the highest order mixed derivative
of the product of $f \circ s$ with some other function $\psi$.
Recall the notation of (\ref{eq:Ek}) and define $E_0 := {\rm id}$.

\begin{lemma}\label{lem:chain}
Suppose that $\psi \in C^n(\Omega)$ and $f \in C^n(s(\Omega))$
for some domain $\Omega \subseteq \RR^n$, and define the product
$$ \phi(\bft) := \psi(\bft)f\big(s(\bft)\big), \qquad \bft \in \Omega. $$
Then,
\begin{equation}\label{eq:Echain}
\frac{\partial^n}{\partial t_1 \cdots \partial t_n} \phi(\bft)
= E_n\phi(\bft)
= \sum_{k=0}^n E_k \psi(\bft) f^{(n-k)}\big(s(\bft)\big), \quad \bft \in \Omega.
\end{equation}
\end{lemma}

\begin{proof}
We show more generally that
\begin{equation}\label{eq:Echainm}
\frac{\partial^m}{\partial t_1 \cdots \partial t_m} \phi(\bft)
= \sum_{k=0}^m
\sum_{1 \le i_1 < \cdots < i_k \le m}
\frac{\partial^k}{\partial t_{i_1} \cdots \partial t_{i_k}}
\psi(\bft) f^{(m-k)}\big(s(\bft)\big)
\end{equation}
for all $m=0,1,\ldots,n$ by induction on $m$.
This is true for $m=0$ by the definition of $\phi$.
Suppose it holds for some $m$, $0 \le m \le n-1$.
Differentiation with respect to $t_{m+1}$ gives
$$ 
\frac{\partial^{m+1}}{\partial t_1 \cdots \partial t_{m+1}} \phi(\bft)
   = a + b, $$
where
\begin{align*}
a &:= \sum_{k=0}^m 
\sum_{1 \le i_1 < \cdots < i_k \le m}
\frac{\partial^{k+1}}{\partial t_{i_1} \cdots \partial t_{i_k} \partial t_{m+1}}
\psi(\bft) f^{(m-k)}\big(s(\bft)\big), \cr
b &:= \sum_{k=0}^m 
\sum_{1 \le i_1 < \cdots < i_k \le m}
\frac{\partial^k}{\partial t_{i_1} \cdots \partial t_{i_k}}
\psi(\bft) f^{(m+1-k)}\big(s(\bft)\big).
\end{align*}
We can write $a$ as
\begin{align*}
a &= \sum_{k=1}^{m+1} 
\sum_{1 \le i_1 < \cdots < i_{k-1} \le m}
\frac{\partial^k}{\partial t_{i_1} \cdots \partial t_{i_{k-1}} \partial t_{m+1}}
\psi(\bft) f^{(m+1-k)}\big(s(\bft)\big) \cr
&= \sum_{k=1}^{m+1} 
\sum_{1 \le i_1 < \cdots < i_k = m+1}
\frac{\partial^k}{\partial t_{i_1} \cdots \partial t_{i_k}}
\psi(\bft) f^{(m+1-k)}\big(s(\bft)\big),
\end{align*}
and adding this to $b$ yields the case $m+1$ of (\ref{eq:Echainm}).
\end{proof}

Now let $V(\bft)$ be the Vandermonde polynomial in~(\ref{eq:vandpoly}).
By Lemma~\ref{lem:Vmixedderivsums},
$E_k V(\bft) = 0$ for all $k=1,\ldots,n$, while $E_0 V(\bft) = V(\bft)$.
Therefore, letting $\psi = V$ in~Lemma~\ref{lem:chain} implies

\begin{lemma}\label{lem:chainV}
Suppose that $f \in C^n(s(\Omega))$ for some domain $\Omega \subseteq \RR^n$.
Then,
$$ V(\bft) f^{(n)}(s(\bft)) =
\frac{\partial^n}{\partial t_1 \cdots \partial t_n}
\big(V(\bft) f(s(\bft))\big), \qquad \bft \in \Omega. $$
\end{lemma}
This observation will simplify the integration in Theorem~\ref{thm:ddid}.

\section{Integration on a rectangle}

Let us now look at integration on a rectangle in $\RR^n$.
Such a rectangle has the form
$$ R := [a_1,b_1] \times [a_2,b_2] \times \cdots 
	     \times [a_n,b_n], $$
and has $2^n$ vertices which we can denote by
$$ \bfv_\bfeps := ((1-\epsilon_1)a_1 + \epsilon_1 b_1,
(1-\epsilon_2)a_2 + \epsilon_2 b_2,
\ldots, (1-\epsilon_n)a_n + \epsilon_n b_n), $$
for $\bfeps = (\epsilon_1,\ldots,\epsilon_n) \in \{0,1\}^n$.
It is easy to integrate the highest order mixed derivative
of a function of $n$ variables on $R$.
Similar to (\ref{eq:s}) we define
$$ s(\bfeps) := \epsilon_1 + \epsilon_2 + \cdots + \epsilon_n. $$
\begin{lemma}\label{lem:intrect}
For any $\phi \in C^n(R)$,
\begin{equation}\label{eq:I}
\int_R
\frac{\partial^n \phi}{\partial t_1 \cdots \partial t_n}
\, dt_1 \cdots dt_n
= \sum_{\bfeps \in \{0,1\}^n}
(-1)^{n - s(\bfeps)} \phi(\bfv_\bfeps).
\end{equation}
\end{lemma}

\begin{proof}
We show more generally that the integral in (\ref{eq:I}) can be written as
\begin{equation}\label{eq:Ik}
\sum_{\bfeps \in \{0,1\}^k}
(-1)^{k - s(\bfeps)}
\int_{a_n}^{b_n} \cdots \int_{a_{k+1}}^{b_{k+1}}
\frac{\partial^{n-k}}{\partial t_{k+1} \cdots \partial t_n}
\phi(\bfv_{\bfeps,k})
\, dt_{k+1} \cdots dt_n
\end{equation}
for any $k=0,1,\ldots,n$, where
$$ \bfv_{\bfeps,k} :=
 ((1-\epsilon_1)a_1 + \epsilon_1 b_1,
\ldots, (1-\epsilon_k)a_k + \epsilon_k b_k,
t_{k+1},\ldots,t_n). $$
The case $k=0$ of (\ref{eq:Ik}) is the left hand side of (\ref{eq:I}).
Suppose (\ref{eq:Ik}) holds for some~$k$, $0 \le k \le n-1$.
Then its innermost integral is
\begin{multline*}
\frac{\partial^{n-k-1}}{\partial t_{k+2} \cdots \partial t_n}
\phi(\bfv_{(\bfeps,1),k+1})
- 
\frac{\partial^{n-k-1}}{\partial t_{k+2} \cdots \partial t_n}
\phi(\bfv_{(\bfeps,0),k+1}) \\
= \sum_{\epsilon_{k+1}=0}^1 (-1)^{1-\epsilon_{k+1}}
\frac{\partial^{n-k-1}}{\partial t_{k+2} \cdots \partial t_n}
\phi(\bfv_{(\bfeps,\epsilon_{k+1}),k+1}),
\end{multline*}
and moving this summation out of the remaining integrals
yields the case $k+1$ of (\ref{eq:Ik}).
Thus (\ref{eq:Ik}) also holds for $k=n$, which
is the right hand side of (\ref{eq:I}).
\end{proof}

Suppose next that $R$ is the sequential rectangle $R(\bfx)$
of (\ref{eq:seq_rect})
and that $\phi$ has the zero property
\begin{equation}\label{eq:zero_property}
\phi(\bft) = 0 \quad \hbox{whenever $t_i = t_{i+1}$
for any $i \in \{1,2,\ldots,n-1\}$.}
\end{equation}
Then the sum in (\ref{eq:I}) is reduced from $2^n$ terms to $n+1$.
\begin{lemma}\label{lem:iid}
If $\phi \in C^n(R(\bfx))$ and $\phi$ has the zero property
(\ref{eq:zero_property}) then
\begin{equation}\label{eq:int_simp}
\int_{R(\bfx)}
\frac{\partial^n \phi}{\partial t_1 \cdots \partial t_n}
\, dt_1 \cdots dt_n
= \sum_{i=1}^{n+1} (-1)^{n+1-i} \phi(\bfv_i),
\end{equation}
where
\begin{equation}\label{eq:vi}
\bfv_i := (x_1,\ldots,x_{n+1-i},x_{n+3-i},\ldots,x_{n+1}).
\end{equation}
\end{lemma}

\begin{proof}
We apply Lemma~\ref{lem:intrect} to the sequential rectangle
$R = R(\bfx)$, whose vertices are
$$ \bfv_\bfeps := ((1-\epsilon_1)x_1 + \epsilon_1 x_2,
(1-\epsilon_2)x_2 + \epsilon_2 x_3,
	\ldots, (1-\epsilon_n)x_n + \epsilon_n x_{n+1}) $$
for $\bfeps = (\epsilon_1,\ldots,\epsilon_n) \in \{0,1\}^n$.
By the zero property (\ref{eq:zero_property}),
$\phi(\bfv_\bfeps) = 0$ whenever $\epsilon_i = 1$ and $\epsilon_{i+1} = 0$
for some $i \in \{1,2,\ldots,n-1\}$.
Thus $\phi(\bfv_\bfeps)$ can only be non-zero if the
sequence $\bfeps$ is non-decreasing,
$\epsilon_1 \le \epsilon_2 \le \cdots \le \epsilon_n$.
There are just $n+1$ such sequences:
\begin{equation}\label{eq:epsdistinct}
\bfeps_i := (\underbrace{0,\ldots,0}_{n+1-i},
	\underbrace{1,\ldots,1}_{i-1}),
\quad i=1,2,\ldots,n+1,
\end{equation}
and the corresponding $n+1$ vertices are
$\bfv_i := \bfv_{(\bfeps_i)}$ which are as in (\ref{eq:vi}).
Furthermore, $s(\bfeps_i) = i-1$.
Therefore, the right hand side of (\ref{eq:I}) reduces to
the right hand side of (\ref{eq:int_simp}).
\end{proof}

\section{Proof of Theorem~\ref{thm:ddid}}

We now complete the proof of Theorem~\ref{thm:ddid}.
Let $\Omega := R(\bfx)$.
Then, by (\ref{eq:yminmax}) and (\ref{eq:tsumbound}),
$s(\Omega) = [y_1,y_{n+1}]$,
and by Lemma~\ref{lem:chainV},
$$
 \int_{R(\bfx)}
V(\bft) f^{(n)}(s(\bft)) \, d\bft = 
 \int_{R(\bfx)}
\frac{\partial^n}{\partial t_1 \cdots \partial t_n}
\big(V(\bft) f(s(\bft))\big) \, d\bft,
$$
and by Lemma~\ref{lem:iid},
\begin{equation}\label{eq:thm2proof}
 \int_{R(\bfx)}
\frac{\partial^n}{\partial t_1 \cdots \partial t_n}
\big(V(\bft) f(s(\bft))\big) \, d\bft =
\sum_{i=1}^{n+1} (-1)^{n+1-i} V(\bfv_i)f(s(\bfv_i)),
\end{equation}
where $\bfv_i$ is as in (\ref{eq:vi}).
We see that $s(\bfv_i) = y_i$ by the
definition of $y_i$ in (\ref{eq:yidef}).
Further, by the definition of the $\bfv_i$ in (\ref{eq:vi}),
the factors in $V(\bfv_i)$ are those of
$V(\bfx)$ except those involving $x_{n+2-i}$,
and we can write
$$ (-1)^{n+1-i} V(\bfv_i) =
\frac{V(\bfx)}{\prod_{j=1, j \ne n+2-i}^{n+1}
(x_j-x_{n+2-i})}. $$
Recalling (\ref{eq:yialt}), we see that
$$ y_i - y_{n+2-j} = x_j - x_{n+2-i}, $$
and so
$$ \prod_{\substack{j=1 \\ j \ne n+2-i}}^{n+1} (x_j-x_{n+2-i})
= \prod_{\substack{j=1 \\ j \ne n+2-i}}^{n+1} (y_i-y_{n+2-j})
= \prod_{\substack{j=1 \\ j \ne i}}^{n+1} (y_i-y_j). $$
Thus the right hand side of (\ref{eq:thm2proof}) can be written as
$$ V(\bfx) \sum_{i=1}^{n+1} 
  \frac{f(y_i)}{\prod_{j=1, j \ne i}^{n+1}
(y_i-y_j)}
= V(\bfx) [y_1,y_2,\ldots,y_{n+1}]f. $$

\Acknowledgement{I am grateful to one of the referees for finding some
misprints in the original version.}

\bibliography{ddvand}

\end{document}